\documentclass[12pt,a4paper]{article}
\usepackage[margin=25mm,headheight=15pt,headsep=8mm]{geometry}
\usepackage[T1]{fontenc}
\usepackage{lmodern}
\usepackage{microtype}

\usepackage{amsmath,amsthm,xcolor}

\usepackage{tikz}
\usetikzlibrary{calc}

\usepackage{hyperref}
\hypersetup{
    colorlinks=true,
    linkcolor=blue,
    citecolor=blue,
    urlcolor=blue
}

\usepackage{cleveref}

\usepackage{caption}
\numberwithin{equation}{section}

\newtheorem{theorem}{Theorem}[section]
\newtheorem{proposition}[theorem]{Proposition}
\newtheorem{corollary}[theorem]{Corollary}
\newtheorem{lemma}[theorem]{Lemma}
\theoremstyle{remark}

\DeclareMathOperator{\area}{area}
\DeclareMathOperator{\length}{length}

\title{\bfseries{A counterexample to Nagamochi's scoring lemma and a new rectangle packing bound}}
\author{
    Hakan Karakuş \\[2pt]
    \small Boğaziçi University, Istanbul, Türkiye \\
    \small\href{mailto:hakan.karakus@bogazici.edu.tr}{\nolinkurl{hakan.karakus@bogazici.edu.tr}}
}
\date{\small September 29, 2026}

\begin{document}

\maketitle

\begin{abstract}
    Let $s(N)$ denote the smallest side length of a square containing $N$ unit squares with arbitrary orientations and pairwise disjoint interiors. Nagamochi's \emph{Packing Unit Squares in a Rectangle} (2005) states a rectangle packing bound from which he deduces two infinite families of exact values: $s(k^2-1) = k$ and $s(k^2-2) = k$ for every integer $k \geq 2$. We construct a family of counterexamples, local to a corner of the container, to the scoring assertion in Nagamochi's Lemma~1. These counterexamples show that the published proof of the rectangle bound is incomplete, but do not disprove the bound itself. We then give an independent proof of a weaker rectangle bound using a strip measure. This recovers $s(k^2-1) = k$ for every integer $k \geq 2$ and yields an explicit lower bound for $s(N)$ that improves strictly on the area bound for every nonsquare integer $N \geq 8$. Our argument does not establish Nagamochi's full rectangle bound or the identity $s(k^2-2) = k$.
\end{abstract}


\section{Introduction}

Finite packing problems ask how small a container of a prescribed shape can be while accommodating a given number of congruent objects with disjoint interiors. Finding a dense arrangement provides an upper bound, but proving its optimality requires ruling out every better arrangement. For equal circles, non-overlap can be expressed through distances between centers; for squares, orientations introduce additional geometric constraints.

For a positive integer $N$, let $s(N)$ denote the smallest side length of a square containing $N$ unit squares with pairwise disjoint interiors, where the squares may be translated and rotated independently. Area and the usual grid packing give
\[
    \sqrt{N} \leq s(N) \leq \left\lceil\sqrt{N}\right\rceil.
\]
For $N = k^2$, the two bounds coincide, giving $s(k^2) = k$.

This paper revisits a rectangle-packing bound stated by Nagamochi \cite{Nagamochi}, from which the identities
\[
    s(k^2-2) = s(k^2-1) = k
\]
were deduced for every integer $k \geq 2$. We give a family of counterexamples to a scoring assertion used in the published proof. We then prove a weaker rectangle bound independently; it recovers $s(k^2-1)=k$ and yields an explicit lower bound for $s(N)$ that improves strictly on the area bound for every nonsquare integer $N \geq 8$.

The square-packing problem considered here goes back at least to Erd\H{o}s and Graham \cite{ErdosGraham}, who showed that rotated squares can substantially reduce the unused area in large containers compared with the usual grid construction. G\"obel \cite{Gobel} subsequently studied the problem systematically, and Friedman \cite{Friedman} surveys its history and known results. Friedman \cite[Section~1]{Friedman} also notes that computational methods developed for circle packing did not readily generalize to squares.

Among the classical exact results, G\"obel \cite{Gobel} obtained $s(5) = 2+\frac{1}{\sqrt{2}}$, Kearney and Shiu \cite{KearneyShiu} proved $s(6) = s(7) = 3$, and Stromquist \cite{Stromquist} proved $s(10) = 3+\frac{1}{\sqrt{2}}$. Friedman \cite{Friedman} proved or reproved several further individual cases, including $N = 8,14,15,24,$ and $35$. Computational methods have also improved constructions without establishing optimality; for example, Gensane and Ryckelynck \cite{GensaneRyckelynck} developed an inflation-based method that improved the best known packings for several values of $N$.

One particularly suggestive sequence occurs immediately below perfect squares. Bentz \cite{Bentz2010} proved $s(13) = 4$ and $s(46) = 7$, and later proved $s(22) = 5$ and $s(33) = 6$ \cite{Bentz}. Together with $s(6) = 3$, these results establish
\[
    s(k^2-3) = k,\qquad k = 3,4,5,6,7.
\]
These individual proofs do not establish the identity for all $k$. Results in the opposite direction show that the grid construction need not remain optimal further below a perfect square: Arslanov, Mustafin, and Shangitbayev \cite{Arslanov} proved
\[
    s(k^2-k) < k,\qquad k \geq 12.
\]
Thus even near perfect squares, determining when rotations permit a smaller container remains difficult.

A rectangle bound stated by Nagamochi \cite{Nagamochi} would provide two infinite families of exact values in this near-square regime. For $t \geq 2$, write
\[
    \Delta(t) = t+1-\lceil t\rceil.
\]
For positive real numbers $a$ and $b$, let $\nu(a,b)$ denote the maximum number of unit squares with pairwise disjoint interiors that can be packed in some $a'\times b'$ rectangle with $a' < a$ and $b' < b$. Nagamochi \cite[Theorem~1]{Nagamochi} states that, for real numbers $a,b \geq 2$,
\begin{equation}\label{eq:nagamochi-rectangle}
    \nu(a,b) < ab-\Delta(a)-\Delta(b).
\end{equation}
From this bound, Nagamochi \cite[Theorem~2]{Nagamochi} derives a general lower bound for $s(N)$ and, in particular,
\begin{equation}\label{eq:close-to-perfect-square}
    s(k^2-1) = s(k^2-2) = k,\qquad k \geq 2.
\end{equation}
These identities have subsequently been cited in the square-packing literature as established results; see, for example, \cite{Friedman,Bentz2010,Bentz,Arslanov}. The validity of Nagamochi's proof therefore matters beyond the individual cases.

Nagamochi's proof assigns scores to squares by means of a central area, four weighted line segments, and weighted points. Lemma~1 of \cite{Nagamochi} asserts a lower bound on the score of each square; summing these scores yields the rectangle bound \eqref{eq:nagamochi-rectangle}. These arguments belong to the method of unavoidable sets: one assigns resources to subsets of the container and shows that every packed square must consume a prescribed amount. Friedman \cite[Sections~4 and~5]{Friedman} surveys constructions using unavoidable points, while Bentz \cite[Section~2]{Bentz} develops continuously varying families of unavoidable point configurations.

We show that Nagamochi's scoring assertion is false under the definitions in \cite[Section~3]{Nagamochi}. More precisely, we identify an edge-incidence condition missing from the application of \cite[Lemma~6]{Nagamochi} in \cite[Section~5.5, Case~6]{Nagamochi}, and use it to construct a family of squares whose scores are less than $1$. The construction applies near a corner of every rectangle with side lengths $a > 3$ and $b > 2$ once its parameter is sufficiently small. A further shrink about a vertex produces a strict configuration in which the contact point lies on an edge different from the two edges cut by $y = 1$.

This counterexample shows that the published proof of \eqref{eq:nagamochi-rectangle} requires an additional argument, but it does not disprove the rectangle bound itself. We instead prove the following weaker bound by a separate method.
\begin{theorem}[Rectangle bound]\label{thm:rectangle-one-direction}
    For real numbers $a \geq 2$ and $b \geq 3$,
    \begin{equation}\label{eq:rectangle-one-direction}
        \nu(a,b) < ab-\Delta(a).
    \end{equation}
    If $a,b \geq 3$, applying the same construction in both directions gives
    \begin{equation}\label{eq:rectangle-symmetric}
        \nu(a,b) < ab-\max\{\Delta(a),\Delta(b)\}.
    \end{equation}
\end{theorem}

The proof constructs a finite nonnegative measure adapted to a horizontal strip, following Nagamochi's general framework of weighted areas, line segments, and points. Line mass on complete horizontal segments and point masses of $1/2$ compensate for area outside the strip. The construction and the required geometric estimates are given in \Cref{sec:replacement}.

Among the consequences in \Cref{sec:consequences} are an obstruction for integer rectangles and an explicit lower bound for $s(N)$ that improves strictly on the area bound for every nonsquare integer $N \geq 8$. In particular, the rectangle bound recovers one of the two near-square identities in \eqref{eq:close-to-perfect-square}.
\begin{corollary}\label[corollary]{cor:minus-one}
    For every integer $k \geq 2$, $s(k^2-1) = k$.
\end{corollary}
This identity is already stated in \cite[Theorem~2(i)]{Nagamochi}; our contribution is a proof for $k \geq 3$ independent of the failed scoring assertion. The case $k = 2$, namely $s(3) = 2$, is classical \cite[Theorem~1]{Friedman}. Our argument does not establish the second identity $s(k^2-2) = k$.

\section{Nagamochi's score}\label{sec:score}

We use the enlarged-square formulation given in \cite[Section 3, following Lemma 1]{Nagamochi}, and write $\sigma(S)$ for Nagamochi's original score. In this formulation, the container and weighted supports are unchanged, while the square being scored has side length $\lambda > 1$. Thus the area and line weights carry no additional factors of $\lambda$.

Let $R = [0,a]\times[0,b]$, where $a,b > 2$ are arbitrary real numbers. We use strict inequalities here so that the eight weighted points in $Q$ below are distinct. The construction in \cite[Section 3 and Figure 1]{Nagamochi} consists of a central rectangle
\[
    R^* = [1,a-1]\times[1,b-1],
\]
four line segments
\begin{equation}\label{eq:area-line-supports}
    \begin{aligned}
        L_1 &= [0.9,a-0.9]\times\{1\},& L_2 &= [0.9,a-0.9]\times\{b-1\},\\
        L_3 &= \{1\}\times[0.9,b-0.9],& L_4 &= \{a-1\}\times[0.9,b-0.9],
    \end{aligned}
\end{equation}
and two sets of weighted points:\footnote{For $P$, we follow \cite[Figure 1]{Nagamochi}, correcting the interchange of $a$ and $b$ in two coordinates of the printed definition.}
\begin{equation}\label{eq:weighted-point-supports}
    \begin{aligned}
        Q &= \{(0.9,j),(a-0.9,j):j \in \{1,b-1\}\}\\
          &\quad{}\cup\{(i,0.9),(i,b-0.9):i \in \{1,a-1\}\},\\[3pt]
        P &= \{(i,0.9),(i,b-0.9):i = 2,\ldots,\lceil a\rceil-2\}\\
          &\quad{}\cup\{(0.9,j),(a-0.9,j):j = 2,\ldots,\lceil b\rceil-2\}.
    \end{aligned}
\end{equation}
An index range is understood to be empty when its upper endpoint is less than $2$. We have $\#Q = 8$ and $\#P = 2\lceil a\rceil+2\lceil b\rceil-12$, where $\#$ denotes cardinality.

Assign area density $1$ to $R^*$, line density $1/2$ to each $L_j$, weight $9/20$ to each point of $Q$, and weight $1/2$ to each point of $P$. The score of a square $S \subseteq R$ is
\begin{equation}\label{eq:square-score}
    \sigma(S) = \area(S\cap R^*)+\frac{1}{2}\sum_{j = 1}^{4}\length(S\cap L_j)+\frac{9}{20}\#(S\cap Q)+\frac{1}{2}\#(S\cap P).
\end{equation}
Here $\length$ denotes Euclidean length. We also use the same score formula for Borel subsets of $R$; since all weights are nonnegative, the score is monotone under inclusion. The four line segments and the eight $Q$-points together contribute
\[
    \frac{1}{2}\left(2\left(a-\frac{9}{5}\right)+2\left(b-\frac{9}{5}\right)\right)+8\cdot\frac{9}{20} = a+b.
\]
Adding the central area and the $P$-point contributions, the total score available in the container is
\begin{equation}\label{eq:total-score}
    (a-2)(b-2)+(a+b)+(\lceil a\rceil+\lceil b\rceil-6) = ab-\Delta(a)-\Delta(b).
\end{equation}
\Cref{fig:nagamochi-construction} shows the full construction in a rectangle with symbolic side lengths $a$ and $b$.

\begin{figure}[!ht]
    \centering
    \begin{tikzpicture}[scale=1.1]
        \pgfmathsetmacro{\rectwidth}{7.4}
        \pgfmathsetmacro{\rectheight}{5.4}
        \fill[gray!10] (1, 1) rectangle (\rectwidth-1, \rectheight-1);
        \draw (-0.2, 0) -- (\rectwidth+0.2, 0);
        \draw (\rectwidth, -0.2) -- (\rectwidth, \rectheight+0.2);
        \draw (\rectwidth+0.2, \rectheight) -- (-0.2, \rectheight);
        \draw (0, \rectheight+0.2) -- (0, -0.2);
        \node[left] at (-0.2, \rectheight) {$b$};
        \node[below] at (\rectwidth, -0.2) {$a$};
        \node[below left] at (-0.2, -0.2) {$0$};
        \node at (\rectwidth/2, \rectheight/2) {$R^*$};
        \draw[ultra thick] (0.9, 1) -- (\rectwidth-0.9, 1);
        \node[above] at (\rectwidth/2, 1) {$L_1$};
        \draw[ultra thick] (0.9, \rectheight-1) -- (\rectwidth-0.9, \rectheight-1);
        \node[below] at (\rectwidth/2, \rectheight-1) {$L_2$};
        \draw[ultra thick] (1, 0.9) -- (1, \rectheight-0.9);
        \node[right] at (1, \rectheight/2) {$L_3$};
        \draw[ultra thick] (\rectwidth-1, 0.9) -- (\rectwidth-1, \rectheight-0.9);
        \node[left] at (\rectwidth-1, \rectheight/2) {$L_4$};
        \foreach \y in {1, \rectheight-1} {
            \filldraw[black] (0.9, \y) circle[radius=1pt];
            \node[left] at (0.9, \y) {$Q$};
            \filldraw[black] (\rectwidth-0.9, \y) circle[radius=1pt];
            \node[right] at (\rectwidth-0.9, \y) {$Q$};
        }
        \foreach \x in {1, \rectwidth-1} {
            \filldraw[black] (\x, 0.9) circle[radius=1pt];
            \node[below] at (\x, 0.9) {$Q$};
            \filldraw[black] (\x, \rectheight-0.9) circle[radius=1pt];
            \node[above] at (\x, \rectheight-0.9) {$Q$};
        }
        \foreach \x in {2, ..., 6} {
            \filldraw[black] (\x, 0.9) circle[radius=1pt];
            \node[below] at (\x, 0.9) {$P$};
            \filldraw[black] (\x, \rectheight-0.9) circle[radius=1pt];
            \node[above] at (\x, \rectheight-0.9) {$P$};
        }
        \foreach \y in {2, ..., 4} {
            \filldraw[black] (0.9, \y) circle[radius=1pt];
            \node[left] at (0.9, \y) {$P$};
            \filldraw[black] (\rectwidth-0.9, \y) circle[radius=1pt];
            \node[right] at (\rectwidth-0.9, \y) {$P$};
        }
    \end{tikzpicture}
    \caption{Nagamochi's scoring construction in a general rectangle, drawn schematically from \cite[Section 3, Figure 1]{Nagamochi}. The central region has area density $1$, the four segments have line density $1/2$, the $Q$-points have weight $9/20$ each, and the $P$-points have weight $1/2$ each.}
    \label{fig:nagamochi-construction}
\end{figure}
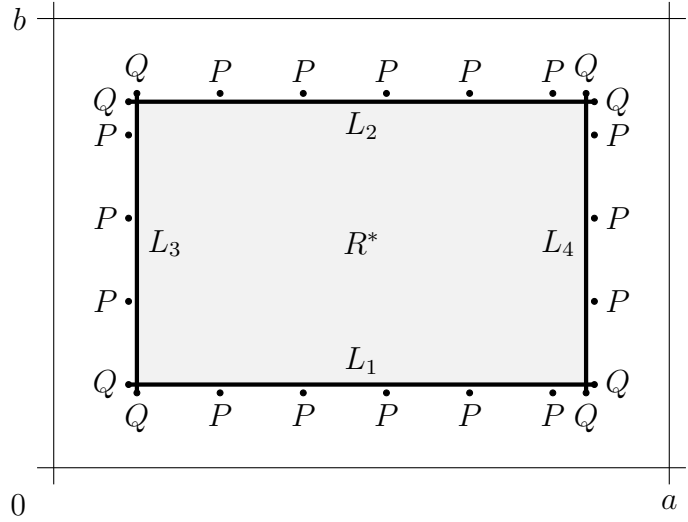

Lemma 1 of \cite{Nagamochi}, in the equivalent formulation just described, requires
\begin{equation}\label{eq:nagamochi-lemma1}
    \sigma(S) > 1\quad\text{ whenever }S \subseteq R\text{ has side length }1 < \lambda \leq \frac{101}{100}.
\end{equation}
This is the local assertion on which the published rectangle bound rests. Choose a packing of $M = \nu(a,b)$ unit squares in an $a'\times b'$ rectangle with $a' < a$ and $b' < b$, and uniformly enlarge it inside $R$. The enlarged squares can then be shrunk slightly about their centers, retaining a common side length $1 < \lambda \leq 101/100$ and making the closed squares pairwise disjoint. If \eqref{eq:nagamochi-lemma1} were valid, summing their scores would give
\[
    \nu(a,b) < \sum_{i = 1}^{M}\sigma(S_i) \leq ab-\Delta(a)-\Delta(b),
\]
by \eqref{eq:total-score}. This explains the intended deduction of \eqref{eq:nagamochi-rectangle} from the scoring assertion.

We next examine the missing edge-incidence condition in \Cref{sec:gap}. The counterexample in \Cref{sec:counterexample} then uses only a small neighborhood of one corner, where the other weighted supports contribute nothing.

\section{The missing edge-incidence condition}\label{sec:gap}

The issue occurs in \cite[Section 5.5, Case 6]{Nagamochi}. We consider $a > 3$ and $b > 2$, the range used for the counterexample below. In Case 6, the center of $S$ lies in $[1,a-1]\times[0,1]$, and the line $y = 1$ intersects two adjacent edges $e_1$ and $e_2$ of $S$. After symmetry and the preceding reductions, $S$ contains $(1,0.9)$, contains no point of $P$, and contains neither $(0.9,1)$ nor $(a-1,0.9)$. To estimate the score in this case, the argument considers a limiting contact configuration with one corner of $S$ on the $x$-axis and $(2,0.9)$ on an edge of $S$. It then handles the case $(1,1) \in S$ separately and invokes \cite[Lemma 6]{Nagamochi} when $(1,1) \notin S$.

However, \cite[Lemma 6]{Nagamochi} requires that the point $(2,0.9)$ lie on the designated edge $e_2$, one of the two adjacent edges intersected by the line $y = 1$, as shown in \cite[Figure 4(b)]{Nagamochi}. Contact with an arbitrary edge need not give this incidence. This suggests examining an almost axis-parallel square for which contact with $(2,0.9)$ occurs on a third edge, distinct from the two cut by $y = 1$, as shown in \Cref{fig:counterexample-idea}.

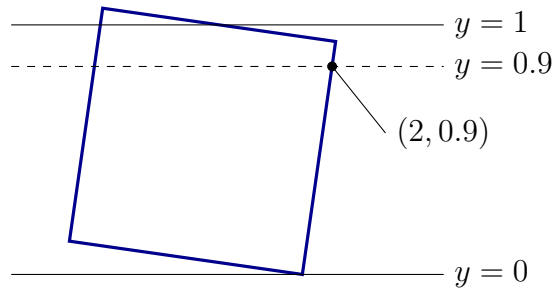
\begin{figure}[htbp]
    \centering
    \begin{tikzpicture}[scale=1.1]
        \coordinate (A) at (3.5, 0);
        \coordinate (B) at (3.9, 2.8);
        \coordinate (C) at (1.1, 3.2);
        \coordinate (D) at (0.7, 0.4);
        \draw[very thick, blue!55!black] (A) -- (B) -- (C) -- (D) -- cycle;
        \draw (0, 0) -- (5.2, 0) node[right] {$y = 0$};
        \draw (0, 3) -- (5.2, 3) node[right] {$y = 1$};
        \draw[dashed] (0, 2.5) -- (5.2, 2.5) node[right] {$y = 0.9$};
        \coordinate (E) at ($(A)!{25/28}!(B)$);
        \draw (E) -- (4.5, 1.7);
        \node[right] at (4.5, 1.7) {$(2, 0.9)$};
        \filldraw[black] (E) circle[radius=1.5pt];
    \end{tikzpicture}
    \caption{The counterexample idea, shown schematically with separations exaggerated. The line $y = 1$ cuts two adjacent edges, while $(2,0.9)$ lies on another edge.}
    \label{fig:counterexample-idea}
\end{figure}

Our counterexample is designed to contain $(1,0.9)$, of weight $9/20$, while capturing nearly one unit of the weighted line $y = 1$. A sufficiently small concentric shrink moves the boundary point $(2,0.9)$ outside while preserving $(1,0.9)$ inside and keeping the side length greater than $1$. Since the shrunken square lies inside the original square, none of the remaining score contributions can increase. The two dominant contributions are then close to
\[
    \frac{9}{20}+\frac{1}{2} = \frac{19}{20}.
\]
With a sufficiently small upper cap and a short intersection with the vertical weighted line, this square can have total score less than $1$. In the next section, we make this construction explicit.

\section{Constructing a counterexample}\label{sec:counterexample}

Fix arbitrary real numbers $a > 3$ and $b > 2$. We now restrict attention to squares near the lower-left corner of $R = [0,a]\times[0,b]$. Specifically, consider squares $S$ satisfying
\begin{equation}\label{eq:local-neighborhood}
    S \subseteq [0,3)\times[0,2),\qquad \max_{(x,y) \in S}x < a-1,\qquad \max_{(x,y) \in S}y < b-1.
\end{equation}
The only weighted points that can meet such a square are
\[
    Q_0 = \{(0.9,1),(1,0.9)\} \subseteq Q,\qquad P_0 = \{(2,0.9)\} \subseteq P.
\]
The segments $L_2$ and $L_4$, on $y = b-1$ and $x = a-1$, are also disjoint from $S$. Thus the complete score \eqref{eq:square-score} reduces in this neighborhood to
\begin{equation}\label{eq:local-score}
    \sigma(S) = \area(S\cap R^*)+\frac{1}{2}\length(S\cap L_1)+\frac{1}{2}\length(S\cap L_3)+\frac{9}{20}\#(S\cap Q_0)+\frac{1}{2}\#(S\cap P_0).
\end{equation}
\Cref{fig:nagamochi-local} shows these local supports.

\begin{figure}[!ht]
    \centering
    \begin{tikzpicture}[scale=2]
        \fill[gray!10] (1, 1) rectangle (2.8, 1.8);
        \draw[->] (-0.1, 0) -- (3, 0) node[right] {$x$};
        \draw[->] (0, -0.1) -- (0, 2) node[above] {$y$};
        \node[below left] at (0, 0) {$0$};
        \draw (1, 0.025) -- (1, -0.025) node[below] {$1$};
        \draw (2, 0.025) -- (2, -0.025) node[below] {$2$};
        \draw (0.025, 1) -- (-0.025, 1) node[left] {$1$};
        \draw[ultra thick, ->] (0.9, 1) -- (2.8, 1);
        \node[below] at (2.6, 1) {$L_1$};
        \draw[ultra thick, ->] (1, 0.9) -- (1, 1.8);
        \node[left] at (1, 1.65) {$L_3$};
        \node at (2.1, 1.5) {$R^*$};
        \filldraw[black] (0.9, 1) circle[radius=0.7pt];
        \node[left] at (0.9, 1) {$Q_0$};
        \filldraw[black] (1, 0.9) circle[radius=0.7pt];
        \node[below] at (1, 0.9) {$Q_0$};
        \filldraw[black] (2, 0.9) circle[radius=0.7pt];
        \node[below] at (2, 0.9) {$P_0$};
    \end{tikzpicture}
    \caption{The relevant part of Nagamochi's scoring construction near the lower-left corner.}
    \label{fig:nagamochi-local}
\end{figure}
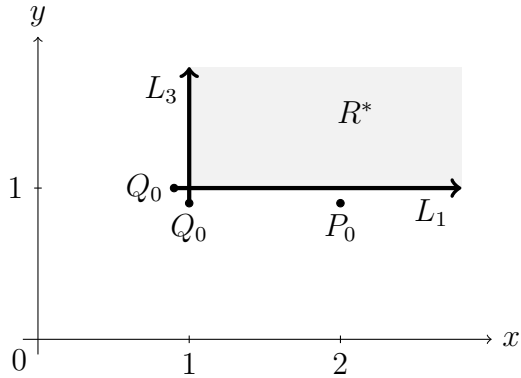

Choose
\begin{equation}\label{eq:counterexample-parameters}
    0 < t \leq \frac{1}{50},\qquad t < \min\{10(a-3),b-2\}.
\end{equation}
The second condition keeps the construction away from the opposite weighted supports. For every such rectangle, admissible values of $t$ exist arbitrarily close to $0$. Let $K_t$ be the contact square with vertices
\begin{equation}\label{eq:contact-square}
    \begin{aligned}
        A &= \left(2-\frac{9}{10}t,0\right), & B &= \left(2+\frac{1}{10}t,1\right), \\
        C &= \left(1+\frac{1}{10}t,1+t\right), & D &= \left(1-\frac{9}{10}t,t\right),
    \end{aligned}
\end{equation}
listed in cyclic order. Its consecutive edge vectors $(t,1)$ and $(-1,t)$ are perpendicular and have equal length. Its side length therefore satisfies
\begin{equation}\label{eq:counterexample-side}
    1 < \lambda_t = \sqrt{1+t^2} \leq \sqrt{1+\frac{1}{2500}} < \frac{101}{100}.
\end{equation}
All its coordinates are nonnegative, and
\begin{equation}\label{eq:contact-bounds}
    \max_{(x,y) \in K_t}x = 2+\frac{t}{10} < \min\{3,a-1\},\qquad \max_{(x,y) \in K_t}y = 1+t < \min\{2,b-1\}.
\end{equation}
Thus $K_t$ satisfies \eqref{eq:local-neighborhood}.

At height $y = 0.9$, its horizontal cross-section is $[1-t^2,2]\times\{0.9\}$; see \eqref{eq:appendix-point-row}. Hence $(1,0.9)$ is in its interior and $(2,0.9)$ lies on $AB$. The vertex $B$ lies exactly on $y = 1$, which simplifies the calculation. \Cref{app:counterexample-calculations} shows that $K_t$ contains exactly one $Q_0$-point, that its intersections with $L_1$ and $L_3$ have lengths $1+t^2$ and $t$, and that its area inside $R^*$ is $t(1+t^2)/2$. Its full score includes the mass $1/2$ of the boundary point $(2,0.9)$. Subtracting this contribution gives
\begin{equation}\label{eq:contact-score}
    \sigma(K_t)-\frac{1}{2} = \frac{19}{20}+t+\frac{1}{2}t^2+\frac{1}{2}t^3.
\end{equation}
The polynomial on the right is strictly increasing for $t > 0$ and equals $242551/250000$ at $t = 1/50$. Therefore,
\begin{equation}\label{eq:kt-score}
    \sigma(K_t\setminus P_0) = \sigma(K_t)-\frac{1}{2} \leq \frac{242551}{250000} < 1.
\end{equation}
For each fixed $t$, shrink $K_t$ about its center by a factor less than $1$ and sufficiently close to $1$, obtaining a square $S_t$ whose side length remains greater than $1$ and whose interior still contains $(1,0.9)$. Since $S_t \subseteq K_t^\circ$, it avoids $P_0$ and satisfies $\sigma(S_t) \leq \sigma(K_t\setminus P_0) < 1$. Its side length is also less than $101/100$ by \eqref{eq:counterexample-side}, so it contradicts \eqref{eq:nagamochi-lemma1}.

To exhibit the missing contact configuration, shrink $K_t$ about $A$ by a factor less than $1$ and sufficiently close to $1$, preserving $(1,0.9)$ in the interior and keeping the side length greater than $1$. The point $(2,0.9)$ remains in the relative interior of the image of $AB$. The images of $B$ and $D$ lie below $y = 1$, while the image of $C$ remains above it. Thus $y = 1$ cuts the images of $BC$ and $CD$, and the contact edge is different from both. In particular, it is not the designated edge $e_2$ required by \cite[Lemma 6]{Nagamochi}. The concentric shrink supplies the score counterexample; the shrink about $A$ serves only to exhibit the missing edge incidence.

\section{A strip measure for rectangles}\label{sec:replacement}

We prove \Cref{thm:rectangle-one-direction} without invoking the scoring assertion or the technical lemmas of \cite{Nagamochi}. Following the same principle of assigning weights to areas, line segments, and points as in \cite[Section 3]{Nagamochi}, we construct a finite nonnegative measure adapted to a horizontal strip. Its total mass is $ab-\Delta(a)$, while the interior of every square of side length slightly greater than $1$ has measure greater than $1$.

Fix $a \geq 2$ and $b \geq 3$, set $R = [0,a]\times[0,b]$, and define
\begin{align*}
    H &= [0,a]\times[1,b-1],\\
    L_- &= [0,a]\times\{1\},\\
    L_+ &= [0,a]\times\{b-1\},\\
    W &= \{(j,4/5),(j,b-4/5):j = 1,\ldots,\lceil a\rceil-1\}.
\end{align*}
Assign area density $1$ to $H$, line density $1/2$ to each of $L_-$ and $L_+$, and point mass $1/2$ to each point of $W$. For a Borel set $E \subseteq R$, define
\begin{equation}\label{eq:strip-measure}
    \mu(E) = \area(E\cap H)+\frac{1}{2}\length(E\cap L_-)+\frac{1}{2}\length(E\cap L_+)+\frac{1}{2}\#(E\cap W).
\end{equation}
Here $\area$ is two-dimensional Lebesgue measure, and $\length$ on each supporting line is one-dimensional Lebesgue measure. The two rows of $W$ are distinct and each contains $\lceil a\rceil-1$ points. Thus $\mu$ is a finite nonnegative measure, with total mass
\begin{equation}\label{eq:replacement-total}
    \mu(R) = a(b-2)+\frac{a}{2}+\frac{a}{2}+\frac{1}{2}(2\lceil a\rceil-2) = ab-\Delta(a).
\end{equation}
\Cref{fig:replacement} shows the construction. For each packed square $S$, we estimate $\mu(S^\circ)$, the mass in its interior.

\begin{figure}[htbp]
    \centering
    \begin{tikzpicture}[scale=1.1]
        \pgfmathsetmacro{\rectwidth}{6.6}
        \pgfmathsetmacro{\rectheight}{4.4}
        \fill[gray!10] (0, 1) rectangle (\rectwidth, \rectheight-1);
        \draw (-0.2, 0) -- (\rectwidth+0.2, 0);
        \draw (\rectwidth, -0.2) -- (\rectwidth, \rectheight+0.2);
        \draw (\rectwidth+0.2, \rectheight) -- (-0.2, \rectheight);
        \draw (0, \rectheight+0.2) -- (0, -0.2);
        \node[left] at (-0.2, \rectheight) {$b$};
        \node[below] at (\rectwidth, -0.2) {$a$};
        \node[left] at (-0.2, \rectheight-1) {$b-1$};
        \node[left] at (-0.2, 1) {$1$};
        \node[below left] at (-0.2, -0.2) {$0$};
        \node at (\rectwidth/2, \rectheight/2) {$H$};
        \draw[ultra thick] (0, 1) -- (\rectwidth, 1);
        \node[above] at (\rectwidth/2, 1) {$L_-$};
        \draw[ultra thick] (\rectwidth, \rectheight-1) -- (0, \rectheight-1);
        \node[below] at (\rectwidth/2, \rectheight-1) {$L_+$};
        \foreach \x in {1, ..., 6} {
            \filldraw[black] (\x, 0.8) circle[radius=1pt];
            \node[below] at (\x, 0.8) {$W$};
            \filldraw[black] (\x, \rectheight-0.8) circle[radius=1pt];
            \node[above] at (\x, \rectheight-0.8) {$W$};
        }
    \end{tikzpicture}
    \caption{The strip measure in a rectangle. The two rows of point masses have integer abscissae strictly between $0$ and $a$, at heights $4/5$ and $b-4/5$.}
    \label{fig:replacement}
\end{figure}
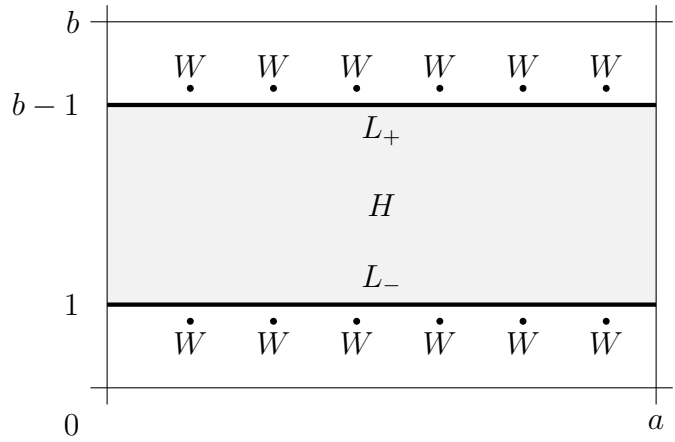

\begin{proposition}\label[proposition]{prop:certificate}
    Every square $S \subseteq R$ of side length $1 < \lambda \leq 101/100$ satisfies $\mu(S^\circ) > 1$.
\end{proposition}
The upper bound $101/100$ is a convenient choice, also used in \cite{Nagamochi}. We first establish elementary geometric estimates. They play roles analogous to the chord and cap estimates in \cite[Lemmas 2--5]{Nagamochi}, but the proofs below are independent of those assertions.

\begin{lemma}[Estimates near a horizontal boundary]\label[lemma]{lem:boundary-estimates}
    Let $S \subseteq \{y \geq 0\}$ be a square of side $1 < \lambda \leq 101/100$ whose center has height at most $1$. Then the following hold.
    \begin{enumerate}
        \item[(i)] For every $r$ satisfying $\frac{3-\sqrt{2}}{2} < r < \sqrt{2}-\frac{1}{2}$,
        \[
            \length(S^\circ\cap\{y = r\}) > 1.
        \]
        \item[(ii)] For $A = \area(S\cap\{y \geq 1\})$ and $\ell = \length(S^\circ\cap\{y = 1\})$,
        \[
            A + \frac{\ell}{2} \geq \lambda^2-\frac{\lambda}{2}.
        \]
    \end{enumerate}
\end{lemma}
\begin{proof}
    Up to reflection and relabeling of the sides, use an orientation $0 \leq \theta \leq \pi/4$. Let
    \[
        h = \sin\theta + \cos\theta\qquad\text{ and }\qquad p = \sin\theta\cos\theta.
    \]
    For $\theta > 0$, the horizontal chord length at height $z$ above the lowest vertex is
    \begin{equation}\label{eq:chord}
        w(z) =
        \begin{cases}
            \dfrac{z}{p},&0 < z < \lambda \sin\theta,\\[6pt]
            \dfrac{\lambda}{\cos\theta},&\lambda \sin\theta \leq z \leq \lambda \cos\theta,\\[6pt]
            \dfrac{\lambda h-z}{p},&\lambda \cos\theta < z < \lambda (\sin\theta+\cos\theta) = \lambda h.
        \end{cases}
    \end{equation}
    \par\medskip
    \noindent\textit{Proof of (i).} Let $z_c$ denote the height of the center of $S$. Since $S \subseteq \{y \geq 0\}$, the vertical distance from the center of $S$ to its lowest point gives $z_c \geq \lambda h/2$. By hypothesis, $z_c \leq 1$.
    
    If $\theta = 0$, then $h = 1$, so $\lambda/2 \leq z_c \leq 1$. The lower and upper sides of $S$ therefore have heights
    \[
        z_c-\frac{\lambda}{2} \leq 1-\frac{\lambda}{2} < \frac{1}{2} < r \qquad\text{ and }\qquad z_c+\frac{\lambda}{2} \geq \lambda > 1 > r,
    \]
    respectively. Thus the line $y = r$ crosses the two vertical sides of $S$, and the resulting chord has length $\lambda > 1$.
    
    Suppose henceforth that $\theta > 0$. Let
    \[
        z_0 = z_c-\frac{\lambda h}{2}
    \]
    be the height of the lowest vertex of $S$. Since
    \[
        \frac{\lambda h}{2} \leq z_c \leq 1,
    \]
    the relative height of the line $y = r$ above this vertex,
    \[
        r-z_0 = r-z_c+\frac{\lambda h}{2},
    \]
    ranges between
    \[
        z_- = \frac{\lambda h}{2}+r-1\qquad\text{and}\qquad z_+ = r.
    \]
    Both values lie strictly between $0$ and $\lambda h$. Since the chord profile \eqref{eq:chord} is concave on its support, it is enough to check these two endpoint positions. We first record two trigonometric bounds.

    We use $p = (h^2-1)/2$ and $1 \leq h \leq \sqrt{2}$, which give
    \begin{equation}\label{eq:trig-bound-1}
        h-p \geq \sqrt{2}-\frac{1}{2} > r,
    \end{equation}
    \begin{equation}\label{eq:trig-bound-2}
        \frac{h}{2}-p \geq \frac{\sqrt{2}-1}{2} > 1-r.
    \end{equation}
    Both expressions on the left are decreasing functions of $h$ on $[1,\sqrt{2}]$, so their minima occur at $h = \sqrt{2}$.

    Since $z_- < \lambda h/2 \leq \lambda\cos\theta$, the height $z_-$ lies in the lower triangle or the plateau. Hence
    \begin{equation}\label{eq:chord-u}
        w(z_-) =
        \begin{cases}
            \dfrac{z_-}{p} > \dfrac{1}{p}\left(\dfrac{h}{2}+r-1\right) > 1\text{ by \eqref{eq:trig-bound-2}}, & 0 < z_- < \lambda \sin\theta,\\[6pt]
            \dfrac{\lambda}{\cos\theta} > \lambda > 1, & \lambda \sin\theta \leq z_- \leq \lambda \cos\theta.
        \end{cases}
    \end{equation}

    Similarly, $\lambda\sin\theta \leq 101/(100\sqrt{2}) < \frac{3-\sqrt{2}}{2} < z_+$, so the height $z_+$ lies in the plateau or the upper triangle. Thus
    \begin{equation}\label{eq:chord-v}
        w(z_+) =
        \begin{cases}
            \dfrac{\lambda}{\cos\theta} > \lambda > 1, & \lambda \sin\theta \leq z_+ \leq \lambda \cos\theta,\\[6pt]
            \dfrac{\lambda h-z_+}{p} > \dfrac{h-r}{p} > 1\text{ by \eqref{eq:trig-bound-1}}, & \lambda \cos\theta < z_+ < \lambda h.
        \end{cases}
    \end{equation}

    Both endpoint chords have length greater than $1$, proving (i).

    \par\medskip
    \noindent\textit{Proof of (ii).} Since $z_c \leq 1$ and the highest point has height $z_c+\lambda h/2 \geq \lambda h > 1$, the line $y = 1$ intersects the square at or above its center. It therefore either cuts two opposite sides or bounds an upper triangular cap.

    If the line cuts two opposite sides, the chord profile \eqref{eq:chord} gives
    \[
        \ell = \frac{\lambda}{\cos\theta},\qquad A = \frac{\lambda^2}{2}-\frac{\lambda}{\cos\theta}(1-z_c) \geq \frac{\lambda^2}{2}-\frac{\lambda}{\cos\theta}\left(1-\frac{\lambda h}{2}\right).
    \]
    Consequently,
    \[
        A+\frac{\ell}{2}-\left(\lambda^2-\frac{\lambda}{2}\right) \geq \frac{\lambda}{2\cos\theta}(\lambda\sin\theta+\cos\theta-1) \geq 0,
    \]
    because $\lambda\sin\theta+\cos\theta \geq \sin\theta+\cos\theta \geq 1$. This also covers the axis-parallel orientation $\theta = 0$.

    If the line creates an upper triangular cap, the chord profile \eqref{eq:chord} and $p > 0$ give
    \[
        \ell = \frac{z_c+\lambda h/2-1}{p} \geq \frac{\lambda h-1}{p},\qquad A = \frac{\ell}{2}\left(z_c+\frac{\lambda h}{2}-1\right) \geq \frac{(\lambda h-1)^2}{2p}.
    \]
    Consequently,
    \[
        A+\frac{\ell}{2}-\left(\lambda^2-\frac{\lambda}{2}\right) \geq \frac{\lambda}{2p}(\lambda-h+p) \geq 0,
    \]
    because $2p = h^2-1$ and $\lambda-h+p > 1-h+p = (h-1)^2/2 \geq 0$.
\end{proof}

For a square $S \subseteq R$ satisfying the lemma's hypotheses, taking $r = 4/5$ in \Cref{lem:boundary-estimates}(i) shows that $S^\circ\cap\{y = 4/5\}$ is an open interval of length greater than $1$. Its set of abscissae is an open interval of length greater than $1$ contained in $[0,a]$, and hence contains an integer in $\{1,\ldots,\lceil a\rceil-1\}$. Thus $S^\circ$ contains at least one weighted point of $W$, contributing $1/2$ to $\mu(S^\circ)$.

\begin{proof}[Proof of \Cref{prop:certificate}]
    Let $z_c$ be the height of the center of $S$.

    \par\medskip
    \noindent\textit{Centers near the bottom or top.} Suppose first that $z_c \leq 1$. Its maximum height is at most
    \[
        z_c+\frac{\lambda}{\sqrt{2}} \leq 1+\frac{101\sqrt{2}}{200} < 2 \leq b-1.
    \]
    Thus $S$ cannot lose weighted area through the top of $H$; its area contribution is exactly $A$ from \Cref{lem:boundary-estimates}(ii). \Cref{lem:boundary-estimates}(i) supplies a weighted point in the interior, contributing mass $1/2$. Combining the two estimates gives
    \[
        \mu(S^\circ) \geq \lambda^2-\frac{\lambda}{2}+\frac{1}{2} = 1+(\lambda-1)\left(\lambda+\frac{1}{2}\right) > 1.
    \]
    Reflection in $y = b/2$ proves the same assertion when $z_c \geq b-1$.

    \par\medskip
    \noindent\textit{Centers in the strip.} It remains to consider $1 \leq z_c \leq b-1$. Call a horizontal boundary of $H$ a \emph{proper cut} if $S$ has positive area on both sides of that line. Every proper cut either removes a triangular cap or meets the interiors of two opposite sides of $S$.

    For a triangular cap with side lengths $u,v \leq \lambda$ measured along two adjacent sides of $S$, let $D$ be its area and $\ell$ its base length. Then
    \begin{equation}\label{eq:cap-compensation}
        \frac{D}{\ell} = \frac{uv}{2\sqrt{u^2+v^2}} \leq \frac{\sqrt{uv}}{2\sqrt{2}} \leq \frac{\lambda}{2\sqrt{2}} < \frac{1}{2}.
    \end{equation}
    Thus the line mass $\ell/2$ more than compensates for the lost area.

    If every proper cut is triangular, compensate every triangular loss using \eqref{eq:cap-compensation}. The area and line contributions are at least $\lambda^2 > 1$.

    If exactly one proper cut is not triangular, its chord has length at least $\lambda$. The center is on the retained side, so at most half the square's area is lost through that boundary. Compensating any triangular loss at the other boundary gives
    \[
        \mu(S^\circ) \geq \frac{\lambda^2}{2}+\frac{\lambda}{2} > 1.
    \]

    If both proper cuts are nontriangular, their line contributions alone total at least $\lambda > 1$.

    These cases include $b = 3$, when a square can cross both horizontal boundaries of the strip. They also account for tangencies, since only proper cuts require compensation. This completes the proof.
\end{proof}

\begin{proof}[Proof of \Cref{thm:rectangle-one-direction}]
    Let $M = \nu(a,b)$, and choose an $a'\times b'$ rectangle with $a' < a$ and $b' < b$ admitting a packing of $M$ unit squares. Choose $\lambda$ sufficiently close to $1$ that
    \[
        1 < \lambda \leq \frac{101}{100},\qquad \lambda a' < a,\qquad \lambda b' < b.
    \]
    Uniformly scaling the packing by $\lambda$ gives squares $S_1,\ldots,S_M$ of side $\lambda$ with pairwise disjoint interiors inside $R$. By additivity and nonnegativity of $\mu$, \Cref{prop:certificate}, and \eqref{eq:replacement-total},
    \[
        M < \sum_{i = 1}^{M}\mu(S_i^\circ) = \mu\left(\bigcup_{i = 1}^{M}S_i^\circ\right) \leq \mu(R) = ab-\Delta(a).
    \]
    This proves \eqref{eq:rectangle-one-direction}. When $a,b \geq 3$, rotating the construction gives $M < ab-\Delta(b)$ as well. Taking the stronger of the two inequalities yields \eqref{eq:rectangle-symmetric}.
\end{proof}

\section{Consequences for rectangle and square packings}\label{sec:consequences}

For integer target side lengths, the correction $\Delta$ equals $1$, giving the following obstruction.
\begin{corollary}[Integer rectangles]\label[corollary]{cor:integer-rectangles}
    Let $m,n \geq 2$ be integers. No $m'\times n'$ rectangle with $m' < m$ and $n' < n$ can contain $mn-1$ unit squares with pairwise disjoint interiors.
\end{corollary}
\begin{proof}
    By interchanging the two directions, assume $n \geq m$. If $n \geq 3$, \Cref{thm:rectangle-one-direction} with $a = m$ and $b = n$ gives $\nu(m,n) < mn-\Delta(m) = mn-1$. If $m = n = 2$, the smaller rectangle is contained in a square of side $\max\{m',n'\} < 2$, which cannot contain three unit squares by the classical identity $s(3) = 2$ \cite[Theorem 1]{Friedman}.
\end{proof}

\begin{proof}[Proof of \Cref{cor:minus-one}]
    Taking $m = n = k$ in \Cref{cor:integer-rectangles} shows that $k^2-1$ unit squares cannot fit in a square of side less than $k$. The $k\times k$ grid with one square removed gives the reverse inequality, so $s(k^2-1) = k$.
\end{proof}

The rectangle theorem also gives an explicit lower bound for every nonsquare integer $N \geq 8$.
\begin{corollary}[A general lower bound]\label[corollary]{cor:replacement-sN-bound}
    Let $N \geq 8$ be a nonsquare integer. Then
    \begin{equation}\label{eq:replacement-sN-bound}
        s(N) \geq \frac{1}{2}+\sqrt{N-\lfloor\sqrt{N}\rfloor+\frac{1}{4}} > \sqrt{N}.
    \end{equation}
\end{corollary}
\begin{proof}
    Let $k = \lfloor\sqrt{N}\rfloor$. Since $N$ is a nonsquare integer, $k^2 < N \leq (k+1)^2-1$. By \Cref{thm:rectangle-one-direction}, any $t \geq 3$ satisfying $t^2-\Delta(t) = N$ gives $s(N) \geq t$. On the interval $k < t \leq k+1$, we have $\Delta(t) = t-k$, so this equation becomes $t^2-t+k = N$. The positive root of this quadratic is
    \[
        t = \frac{1}{2}+\sqrt{N-k+\frac{1}{4}}.
    \]
    The inequalities for $N$ give
    \[
        \left(k-\frac{1}{2}\right)^2 < N-k+\frac{1}{4} \leq \left(k+\frac{1}{2}\right)^2
    \]
    and therefore $k < t \leq k+1$. Hence $\lceil t\rceil = k+1$, and by construction
    \[
        t^2-\Delta(t) = t^2-t+k = N.
    \]

    We also have $t \geq 3$: for $N = 8$, the formula gives $t = 3$, while for $N > 8$ the nonsquare assumption implies $k \geq 3$ and hence $t > k \geq 3$. If $N$ unit squares could be packed in a square of side less than $t$, then \Cref{thm:rectangle-one-direction}, applied with $a = b = t$, would give
    \[
        N \leq \nu(t,t) < t^2-\Delta(t) = N,
    \]
    a contradiction. Thus $s(N) \geq t$.

    Finally, since $t > k$,
    \[
        t^2 = N+t-k > N,
    \]
    and hence $t > \sqrt{N}$.
\end{proof}

For $N = n^2-1$ with $n \geq 3$, the lower bound in \eqref{eq:replacement-sN-bound} is exactly $n$. Thus the family in \Cref{cor:minus-one} also occurs at the endpoints of this general bound.

\section{Concluding remarks}\label{sec:scope}

\Cref{sec:counterexample} gives a local family violating the scoring assertion in \cite[Lemma~1]{Nagamochi}. The family applies near a corner of every rectangle with $a > 3$ and $b > 2$ once the parameter is sufficiently small. Shrinking the same contact square about a vertex produces a strict configuration exhibiting the missing edge-incidence condition in the application of \cite[Lemma~6]{Nagamochi}. The published derivation of the rectangle bound and its consequences therefore requires a correction or an additional argument; the family is not itself a packing that violates those bounds.

The strip measure provides a partial replacement for Nagamochi's theorem. For $a,b \geq 3$, it yields the correction $\max\{\Delta(a),\Delta(b)\}$ in place of $\Delta(a)+\Delta(b)$. Its consequences include the integer-rectangle obstruction, the lower bound \eqref{eq:replacement-sN-bound}, and an independent proof of
\[
    s(k^2-1) = k.
\]
Thus one of the two nontrivial infinite near-square families stated in \cite[Theorem~2]{Nagamochi} is recovered without the disputed scoring assertion. The present argument does not establish or disprove Nagamochi's full rectangle bound, the identity
\[
    s(k^2-2) = k,
\]
or the stronger general lower bound stated in \cite[Theorem~2]{Nagamochi}.

This leaves a natural geometric question. The strip measure charges only one pair of opposite sides of the container and has total mass $ab-\Delta(a)$; after rotation it gives the corresponding correction in the other direction. Can the two boundary corrections be combined without creating a square of mass at most $1$? For a square container $[0,k]^2$, such a construction would need total mass $k^2-2$ rather than the $k^2-1$ supplied by the present strip measure. A successful construction of this kind could provide a new route toward the unresolved part of Nagamochi's claimed bound.

More broadly, the issue illustrates the difficulty of rigorous lower bounds in square packing. Numerical and geometric constructions can often suggest very efficient arrangements, whereas an optimality proof must control every possible translation and orientation. Weighted measures and unavoidable-set arguments remain useful precisely because they turn this global packing problem into local geometric inequalities, but the counterexample here shows that the incidence conditions underlying those inequalities must be tracked carefully.

\paragraph{Use of AI tools.} AI tools were used during exploratory work, for symbolic and numerical checks, and for language and typesetting assistance. All mathematical arguments and calculations appearing in the article were independently checked by the author.

\appendix
\crefalias{section}{appendix}

\section{Calculations for the counterexample}\label{app:counterexample-calculations}

We verify the contact-square score \eqref{eq:contact-score} used in \Cref{sec:counterexample}. Throughout this appendix, $a > 3$, $b > 2$, and $t$ satisfies \eqref{eq:counterexample-parameters}. The square $K_t$ has vertices \eqref{eq:contact-square}, with edge directions $B-A = C-D = (t,1)$ and $D-A = C-B = (-1,t)$.

\paragraph{Weighted point contributions.}

Parametrize $AB$ by $A+s(t,1)$, where $0 \leq s \leq 1$. At height $y = 0.9$, we have $s = 9/10$, giving the right endpoint
\[
    \left(2-\frac{9}{10}t\right)+\frac{9}{10}t = 2.
\]
Similarly, on $DC$, parametrized by $D+s(t,1)$, the same height gives $s = 9/10-t$ and the left endpoint
\[
    \left(1-\frac{9}{10}t\right)+t\left(\frac{9}{10}-t\right) = 1-t^2.
\]
Both parameter values lie strictly between $0$ and $1$, so the chord is
\begin{equation}\label{eq:appendix-point-row}
    K_t\cap\{y = 0.9\} = [1-t^2,2]\times\{0.9\}.
\end{equation}
Thus $(1,0.9)$ lies in $K_t^\circ$, and $(2,0.9)$ lies on $AB$. Moreover,
\[
    \min_{(x,y) \in K_t}x = 1-\frac{9}{10}t \geq \frac{491}{500} > 0.9,
\]
so the point column $x = 0.9$ is disjoint from $K_t$. Together with \eqref{eq:contact-bounds}, this excludes all remaining weighted points. Hence
\begin{equation}\label{eq:appendix-point-counts}
    \#(K_t\cap Q_0) = 1,\qquad \#(K_t\cap P_0) = 1,
\end{equation}
and the total point contribution is $9/20+1/2$.

\paragraph{Intersection with the horizontal weighted line.}

The vertex $B$ lies on $y = 1$. On the opposite edge $DC$, the parametrization $D+s(t,1)$ reaches this height at $s = 1-t$. Its abscissa is
\[
    \left(1-\frac{9}{10}t\right)+t(1-t) = 1+\frac{1}{10}t-t^2.
\]
The full chord lies within $L_1$, giving
\begin{equation}\label{eq:appendix-horizontal}
    K_t\cap L_1 = \left[1+\frac{1}{10}t-t^2,2+\frac{1}{10}t\right]\times\{1\}.
\end{equation}
Subtracting the endpoints yields
\begin{equation}\label{eq:appendix-horizontal-length}
    \length(K_t\cap L_1) = 1+t^2.
\end{equation}
The left endpoint lies strictly to the right of $x = 1$, since $t/10-t^2 > 0$. In particular, $(1,1)$ is outside $K_t$.

\paragraph{Intersection with the vertical weighted line.}

The upper intersection with $x = 1$ lies on $DC$. Its first coordinate equals $1$ when
\[
    1-\frac{9}{10}t+st = 1,\qquad s = \frac{9}{10},
\]
giving height $0.9+t$. The lower intersection lies on $AD$, parametrized by $A+s(-1,t)$. Its first coordinate equals $1$ when
\[
    2-\frac{9}{10}t-s = 1,\qquad s = 1-\frac{9}{10}t,
\]
giving height $t(1-9t/10) < t \leq 1/50 < 0.9$. Since $L_3$ starts at height $0.9$ and extends beyond $0.9+t$, it follows that
\begin{equation}\label{eq:appendix-vertical}
    K_t\cap L_3 = \{1\}\times[0.9,0.9+t],\qquad \length(K_t\cap L_3) = t.
\end{equation}

\paragraph{Area inside the central region.}

The upper vertex $C$ lies to the right of $x = 1$, as does the full chord \eqref{eq:appendix-horizontal}. Together with \eqref{eq:contact-bounds}, this shows that the portion above $y = 1$ is entirely inside $R^* = [1,a-1]\times[1,b-1]$. It is a triangle with base $1+t^2$ and height $t$, so
\begin{equation}\label{eq:appendix-area}
    \area(K_t\cap R^*) = \frac{t}{2}(1+t^2).
\end{equation}

\paragraph{The contact-square score.}

Substituting \eqref{eq:appendix-point-counts}, \eqref{eq:appendix-horizontal-length}, \eqref{eq:appendix-vertical}, and \eqref{eq:appendix-area} into \eqref{eq:local-score} gives
\[
    \sigma(K_t) = \frac{9}{20}+\frac{1}{2}+\frac{1}{2}(1+t^2)+\frac{t}{2}+\frac{t(1+t^2)}{2} = \frac{29}{20}+t+\frac{1}{2}t^2+\frac{1}{2}t^3.
\]
Subtracting the boundary point's mass $1/2$ yields \eqref{eq:contact-score}.

\end{document}